\documentclass[11pt]{amsart}

\usepackage[margin=1in]{geometry}
\usepackage{amsmath,amssymb,amsthm,mathtools}
\usepackage{mathrsfs}
\usepackage{enumitem}
\usepackage[colorlinks=true,linkcolor=blue,citecolor=blue,urlcolor=blue]{hyperref}
\usepackage[nameinlink,noabbrev]{cleveref}
\usepackage{graphicx}
\usepackage{caption}

\newtheorem{theorem}{Theorem}

\newtheorem{lemma}{Lemma}
\newtheorem{remark}{Remark}

\newcommand{\T}{\mathbb T}
\newcommand{\R}{\mathbb R}

\newcommand{\diver}{\operatorname{div}}

\begin{document}
\title{Turbulent Flame Speed Can Increase under Curvature Smoothing}
\author[H. V. Tran, J. Xin, Y. Yu]{Hung V. Tran, Jack Xin, Yifeng Yu}

\date{}

\thanks{
H. V. Tran is partially supported by NSF grant DMS-2348305. 
J. Xin is partially supported by NSF grant DMS-2309520.
}
\address[H. V. Tran]
{
Department of Mathematics, 
University of Wisconsin-Madison, Van Vleck Hall, 480 Lincoln Drive, Madison, Wisconsin 53706, USA}
\email{hung@math.wisc.edu}

\address[J. Xin]
{
Department of Mathematics, 
University of California at Irvine, 
California 92697, USA}
\email{jack.xin@uci.edu}

\address[Y. Yu]
{
Department of Mathematics, 
University of California at Irvine, 
California 92697, USA}
\email{yifengy@uci.edu}

\begin{abstract}
Curvature effects are expected to smooth flame-front wrinkles and thereby reduce turbulent flame speed. We construct a smooth three-dimensional periodic shear flow for which introducing Markstein curvature diffusivity instead increases the effective flame speed predicted by the level-set G-equation. This gives the first counterexample, within this model, to monotone slowdown under curvature smoothing and contrasts with the rigorous monotonicity result for two-dimensional shear flows. The example reveals a genuinely multidimensional mechanism in which local curvature smoothing can enhance, rather than suppress, large-scale front propagation.
\end{abstract}

\maketitle

\section{Introduction}
The G-equation is a well-known model in turbulent combustion
\cite{M1951,M1964,Peters2000,W1985}. It is formulated as the level-set equation
\[
    G_t+s_l|DG|+V(x)\cdot DG=0
    \qquad \text{in } \mathbb{R}^n\times (0,\infty).
\]
For each time $t>0$, the set $\{G>0\}$ represents the unburned region, while
the set $\{G<0\}$ represents the burned region. Here $V$ denotes the velocity
field of the ambient fluid, and $s_l$ is the laminar flame speed, that is, the
local burning velocity. See Figure \ref{fig:front}.

When curvature effects are taken into account, the local burning velocity is 
modified to

\[
    s_l=s_L(1-d\delta_L\, \kappa)_+ .
\]
Here, $s_L$ is a positive constant,  $d$ is the Markstein number, $\delta_L$ the flame thickness, and $\kappa$ denotes the mean curvature of the
flame front. The curvature correction term $d\delta_L\, \kappa$ was introduced in \cite{M1951} as a
phenomenological model for the variation of temperature along the flame front
\cite{M1951}. The cutoff $(\cdot)_+$ is imposed to avoid negative local burning velocity because burned material cannot
become unburned \cite{ZR1994}. For background on level-set methods and the rigorous mathematical
theory of curvature-driven equations, including the relevant notions
of viscosity solutions, we refer to
\cite{OS,EvansSpruck1991,ChenGigaGoto1991,OsherFedkiw2003,
Evans2010,tran2021}.

One important application of the G-equation is to model the turbulent flame
speed, namely the effective or averaged burning velocity induced by the ambient
flow and the geometry of the flame front. Turbulent flame speed, also called the effective burning velocity in the homogenization literature, is one of the central quantities in turbulent combustion \cite{Peters2000}. 

Under the homogenization scaling \(V(x)\mapsto V(x/\varepsilon)\),
\(\varepsilon\) represents the microscopic flow scale. Assuming
\(\delta_L\sim\varepsilon\), the curvature coefficient \(d\delta_L\) is
of order \(d\varepsilon\). For an initially planar front
\(P\cdot x=0\), the effective burning velocity in the direction
\(P\in\mathbb{R}^3\) is the asymptotically constant macroscopic
propagation rate, characterized by the ergodic constant
\(\widetilde c(d,P)\) of the associated cell problem. We refer to the survey paper \cite{XYR2024} for further discussion of the physical background, corresponding homogenization theory, and related mathematical results.  Throughout this paper, the terms
\emph{turbulent flame speed} and \emph{effective burning velocity} are
used interchangeably for the homogenized macroscopic propagation rate.

The existence of an effective burning velocity for general two-dimensional periodic incompressible flows was proved
in \cite{GZXY2024} using a game-theoretic approach. 

In three dimensions, however, the effective burning velocity may fail to exist even for simple shear
flows when the flow intensity exceeds a certain threshold \cite{MMTXY2025}.

A natural and important question is how the predicted effective burning
velocity depends on the curvature effect, in particular on the Markstein number
$d$. This question has been widely studied in the combustion literature, 
where the prevailing intuition is that positive
Markstein curvature effects suppress flame-front wrinkling and thereby
tend to reduce the turbulent burning velocity \cite{R1995}. See \cite{ChaudhuriWuLaw2013} for experimental results. This monotonicity was rigorously proved for two-dimensional shear flows in \cite{LXY2018}.

In this paper,  we construct a three-dimensional example
in which the curvature term increases this effective speed, thereby
answering \cite[Question~8]{XYR2024} in the negative. Thus, the
effective burning velocity need not decrease monotonically with the
Markstein number.


\begin{figure}
\centering
\includegraphics[width=8.7cm]{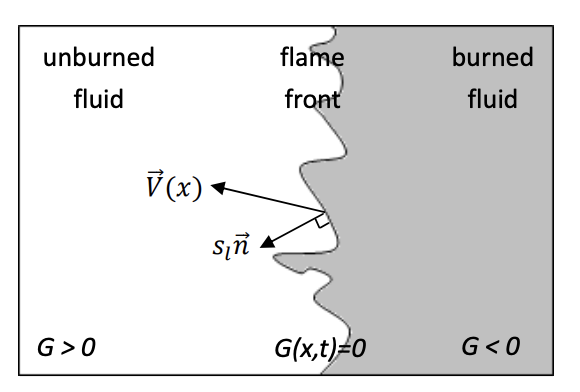}
\caption{Level-set formulation of the front propagation of the G-equation.}
\label{fig:front}
\end{figure}

Without loss of generality, we normalize \(s_L=1\). Fix \(p\in\mathbb{R}^2\), and consider the three-dimensional shear flow
\[
V(x_1,x_2,x_3)=\bigl(0,0,f(x_1,x_2)\bigr), \qquad \text{where } f\in C^\infty(\T^2).
\]

For \(d\geq 0\), the effective burning velocity in the direction \((p,1)\) is characterized, under homogenization, by the ergodic constant \(c(d)\), when it exists, for which the cutoff cell problem
\begin{equation}\label{eq:cell-problem-cutoff}
\left(
1-d\operatorname{div}\left(
\frac{p+D\tilde v}{\sqrt{1+|p+D\tilde v|^2}}
\right)
\right)_{+}
\sqrt{1+|p+D\tilde v|^2}
+f(x)
=
\tilde c(d)
\qquad \text{in }\mathbb{T}^2
\end{equation}
admits a periodic corrector \(\tilde v\). Here and below, we suppress the dependence of \(\tilde c(d)\) on \((p,1)\).

To analyze this problem, we also consider the corresponding non-cutoff cell equation
\begin{equation}\label{eq:cell-problem}
\left(
1-d\operatorname{div}\left(
\frac{p+Dv}{\sqrt{1+|p+Dv|^2}}
\right)
\right)
\sqrt{1+|p+Dv|^2}
+f(x)
=c(d)
\qquad \text{in }\mathbb{T}^2.
\end{equation}
The existence of a periodic solution \(v\) and the associated constant \(c(d)\) for this equation was established in~\cite{GZXY2024}. Moreover, whenever
\[
 c(d)\geq \max_{\mathbb{T}^2} f,
\]
the cutoff is inactive. Consequently, the same pair \((v,c(d))\) also solves the cutoff cell problem (\ref{eq:cell-problem-cutoff}), and \(c(d)=\tilde c(d)\) is therefore the effective turbulent burning velocity for the shear flow \(V\) in the direction \((p,1)\).

\begin{theorem}\label{theo:main} 
For $p=(2,2)$, there exists \(f\in C^{\infty}(\T^2)\) such that

{\rm(i)} $\max_{\T^2}f<0$;

{\rm(ii)} $c(0)=0$;

{\rm(iii)} the following limit exists and is positive:
\[\lim_{d\to 0^{+}}\frac{c(d)}{d}>0.
\]
\end{theorem}

\begin{remark}\label{rmk:tocutoff}
    
As an immediate corollary of Theorem \ref{theo:main}, we get that when $d$ is small enough, $c(d)>0$, and $c$ is not monotonically decreasing in a neighborhood of $0$. 
Note that when $d$ is small enough, $c(d)-\max_{\T^2}f>0$ and thereby $c(d)=\tilde c(d)$. 
Accordingly, the above result shows that the turbulent flame speed need not be a monotonically decreasing function of the Markstein diffusivity $d$. This contrasts with the usual expectation in the combustion literature, as well as with the mathematical result known for two-dimensional shear flows \cite{LXY2018}.
Specifically, Theorem \ref{theo:main} 
gives a negative answer to \cite[Question 8]{XYR2024}. Also, since the ergodic constant \(c(d)\) depends continuously
on \(f\) in the \(C^0\)-norm, for a fixed $d_*>0$, the strict inequality
\(c(d_*)>c(0)>\max_{\T^2}f\) persists under sufficiently small \(C^0\)
perturbations of \(f\). Thus, the failure of monotonicity is robust.

\end{remark}

\begin{remark}\label{rmk:non-zero-d}
The failure of monotonicity is not restricted to the behavior near $d=0$. 
For example, when $d$ is close to $1$, one can construct another example, based on a different mechanism, for which $c(d)$ in \eqref{eq:cell-problem} is not monotonically decreasing. 
See Appendix \ref{sec:appendix2}.
\end{remark}

\subsection*{Relation to weak-flow perturbation theory}

Consider a weak perturbation of a constant flow,
\[
V_\delta=V_0+\delta V_1,
\qquad 0<\delta\ll1.
\]
Previous  asymptotic analysis
\cite[Section~2.2]{LXY2018}, under a Diophantine condition on the
propagation direction \(P\), predicts that for each fixed \(d>0\),
\[
c_\delta(d,P)<c_\delta(0,P)
\]
when \(\delta\) is sufficiently small. It remains open whether, for a
fixed sufficiently small \(\delta>0\), this slowdown inequality persists
as \(d\to0^+\) in a Diophantine direction. Addressing this question
requires constructing smooth correctors through a small-divisor
perturbative scheme and controlling the resulting error estimates
uniformly as \(d\to0^+\); it may therefore be viewed as a KAM-type
problem.

Our constructed family is also a perturbation of a constant flow, but
depends nonlinearly and smoothly on a parameter \(\delta\). Although the
corresponding direction \(P=(2,2,1)\) is rational, a direct expansion in
\(\delta\), for each fixed \(d>0\), gives
\[
c_\delta(d,P)<c_\delta(0,P)=0
\]
when \(\delta\) is sufficiently small. On the other hand, our main
theorem shows that, for every fixed sufficiently small \(\delta>0\),
\[
c_\delta(d,P)>c_\delta(0,P)
\]
when \(d>0\) is sufficiently close to zero. Thus, either sign can occur
within the same family, depending on the relative sizes of \(d\) and
\(\delta\). In particular, the fixed-\(d\) expansion is not uniform as
\(d\to0^+\), and the small-amplitude and vanishing-curvature limits need
not commute.

\subsection*{Acknowledgment} We would like to mention that ChatGPT 5.5 Plus played a significant role in the development of the proof. We discuss this in more detail in the next section.

\section{Exploratory Journey
Leading to the Proof}

In this section, we briefly describe how the problem and its solution were explored. 
Initially, we tried to prove monotonicity, which is a natural direction given the common expectation in the combustion literature. A natural approach is to differentiate the cell problem \eqref{eq:cell-problem} with respect to the Markstein number $d$. Formally, this leads to a linearized equation for $w=\partial_d v$, which is of the form
\begin{equation}\label{eq:linearized-d}
-\operatorname{tr}\big(A(Dv,D^2v)D^2w\big)
+
B(Dv,D^2v)\cdot Dw
=
c'(d).
\end{equation}
The hope was that this linearized equation might have some special structure forcing the constant $c'(d)$ on the right-hand side to have a definite sign. However, there seems to be no clear mathematical reason why such a sign condition should hold in general.

In two-dimensional shear flows, $c'(d)$ was rigorously shown to be negative in \cite{LXY2018}, 
since the cell problem reduces to a one-dimensional ordinary differential equation (ODE), and $c'(d)$ admits an explicit, although somewhat sophisticated, expression in terms of $w$. This makes the analysis much more accessible. Such a reduction is not available if $n=3$, when the cell problem is genuinely two-dimensional and considerably more complicated.

When we first asked several large language models, including ChatGPT 5.5 Plus, about this question, most of them followed the same natural strategy: they differentiated the equation with respect to $d$ and then tried to use a maximum-principle argument to determine the sign of $c'(d)$. These attempts, however, did not lead to a proof of monotonicity.

To simplify the issue, instead of considering all $d$, we focused on the behavior near $d=0$. 
We note that  \eqref{eq:cell-problem} is not necessarily differentiable at $d=0$, since the inviscid equation
\[
\sqrt{1+|p+Dv|^2}+f(x)=c(0)
\qquad \text{in } \mathbb{T}^2
\]
may have multiple nonsmooth solutions (see \cite[Section 4.8]{tran2021}). 
Nevertheless, if one formally differentiates \eqref{eq:cell-problem} at $d=0$, one obtains a linear transport equation of $w$, which might not be well-posed:
\begin{equation}\label{eq:formal-linearized-zero}
\frac{Dw\cdot (p+Dv)}{\sqrt{1+|p+Dv|^2}}
-
\sqrt{1+|p+Dv|^2}\,
\operatorname{div}\left(
\frac{p+Dv}{\sqrt{1+|p+Dv|^2}}
\right)
=
c'(0).
\end{equation}
Now let $u(x)=p\cdot x+v(x)$,
\begin{align*}
W(x)
=
\frac{Du(x)}{\sqrt{1+|Du(x)|^2}},
\qquad
g(x)
=
\sqrt{1+|Du(x)|^2}\,
\operatorname{div}
\left(
\frac{Du(x)}{\sqrt{1+|Du(x)|^2}}
\right).
\end{align*}
Consider a trajectory $\xi:[0,T]\to \mathbb{R}^2$ associated with the vector field $W$, namely
\[
    \dot{\xi}(t)=W(\xi(t)),
\qquad \text{such that} \qquad
    \xi(T)-\xi(0)\in \mathbb{Z}^2 .
\]
Integrating \eqref{eq:formal-linearized-zero} along $\xi$ over $[0,T]$, and using the periodicity of $w$, leads formally to
\begin{equation}\label{eq:orbit-average}
    -\frac{1}{T}\int_0^T g(\xi(t))\,dt = c'(0).
\end{equation}
We point out that the above equality is in general not correct since the left-hand side depends on the choice of $\xi$. A correct formula in our example will be given later. 

\subsection*{Remark on the Role of ChatGPT 5.5 Plus}
At this point, we started searching for an example where $c'(0)$ is positive. In this step, we received useful assistance from ChatGPT 5.5 Plus. The interaction was centered around two main directions: first, to look for a choice of smooth $\mathbb{Z}^2$-periodic $v$ and a corresponding trajectory $\xi$ for which the orbit integral \[ \int_0^T g(\xi(t))\,dt \] is negative; and second, to explore how the preceding formal calculation could be turned into a rigorous argument.

Through a sequence of interactions, ChatGPT 5.5 Plus played a substantial role in suggesting the main formal steps leading to the construction. It helped identify promising ansatzes, organize the linearized calculation, and formulate the orbit-average mechanism that indicated how a positive value could arise. The resulting ideas were then carefully verified, refined, and made rigorous by the authors.

Heuristically, monotonicity holds in the two-dimensional case because the cell problem reduces to a one-dimensional profile equation. In that setting, the curvature term acts as a genuine smoothing mechanism: increasing the Markstein number suppresses front wrinkling and therefore decreases the shear-induced enhancement of the effective burning velocity. However, this intuition relies strongly on the resulting one-dimensional structure. In the genuinely two-dimensional case of the cell problem, there is no comparable ODE reduction, and in particular, there is no diffusion along the tangential direction in the same effective sense, so that the degenerate nature of curvature smoothing unveils itself. 
Our counterexample shows that the one-dimensional monotonicity intuition in the two-dimensional shear flows does not extend to the corresponding cell problem for three-dimensional shear flows.


\subsection*{Connection with vanishing-viscosity selection}

The orbit-average mechanism found here is reminiscent of selection
principles in vanishing-viscosity Hamilton--Jacobi equations
\cite{AIPS2005}. When the Aubry set consists of periodic orbits, the
leading-order effect of a small viscosity can be governed by extremal
averages of the viscous correction along those orbits. In the present
curvature problem, the right derivative of the ergodic constant at
\(d=0\) is likewise determined by an extremal average of the curvature
contribution along periodic characteristics. This suggests a broader
principle: for first-order Hamilton--Jacobi dynamics perturbed by a
small, possibly degenerate, second-order operator, the macroscopic
effect of the regularization may be controlled by orbit selection
rather than by pointwise smoothing alone.

\section*{Summary and Outlook}

To our knowledge, we construct the first counterexample in three
dimensions, within the curvature \(G\)-equation, showing that the
prevailing expectation of monotone slowdown under curvature smoothing
need not hold. Although the curvature term acts locally as a smoothing
mechanism, its effect on the large-scale propagation speed can have the
opposite sign in a genuinely multidimensional setting. Together with
the fixed-\(d\) weak-flow expansion, our result shows that curvature may
either decrease or increase the effective burning velocity, depending
on the relative sizes of the flow perturbation and the Markstein
diffusivity. The acceleration found here emerges in the
vanishing-curvature regime. The development of the example also
illustrates how advanced AI tools can assist mathematical exploration,
with all resulting claims and proofs independently verified and
established rigorously by the authors.

\section{Proof of Theorem \ref{theo:main}}
Fix
\[
p=(2,2).
\]
Let \(a:\mathbb R\to \mathbb R\) be a smooth $1$-periodic function such that $\|a'\|_{L^\infty}\leq 1$.
For \(x=(x_1,x_2)\in \mathbb R^2\), define
\begin{equation}\label{eq:v-and-u}
    v(x_1,x_2)=a(x_1)+a(x_2),\qquad
    u(x)=v(x)+p\cdot x.
\end{equation}
Then,
\[
    Du(x)=\bigl(2+a'(x_1),\,2+a'(x_2)\bigr).
\]
As $\|a'\|_{L^\infty}\leq 1$, $|Du(x)|\geq \sqrt{2}$ for $x\in \R^2$.
Let
\[
    f(x)=-\sqrt{1+|Du(x)|^2}.
\]
Then, $f<0$ and $v$ is a periodic smooth solution to
\begin{equation}
    \sqrt{1+|p+Dv|^2}+f(x)=0  \qquad \text{in } \mathbb{T}^2.
\end{equation}
Hence,
\[
c(0)=0.
\]

It remains to establish item (iii) in Theorem \ref{theo:main}.
Note that every gradient flow $\eta(t)$ of \(u\) is of the form
\[
    \eta(t)=\bigl(s(t),s(t+t_0)\bigr) \quad \text{for $t\in \mathbb{R}$},
\]
where \(s(t)\) solves
\[
    s'(t)=2+a'(s(t)), \qquad t\in\mathbb R,
\]
and \(t_0\in\mathbb R\) is fixed. Since \(s(t)\) increases by one over each period of the corresponding flow, the projection of this curve onto \(\mathbb T^2\) is periodic, with rotation vector \((1,1)\).
Recall that
\[
 W(x)=\frac{Du(x)}{\sqrt{1+|Du(x)|^2}},\qquad
g(x) =\sqrt{1+|Du(x)|^2}\diver \left(\frac{Du(x)}{\sqrt{1+|Du(x)|^2}}\right).
\]
Denote by $\Gamma$  the set of all flows:
\[
 \Gamma=\{\xi:\mathbb{R}\to \mathbb{R}^2\,:\,    \dot \xi(t)=W(\xi(t)) \text{ for } t\in \mathbb R\}.
\]
Since each such curve $\xi$ is a reparametrization of a gradient flow line of $u$, its projection onto the flat torus $\mathbb T^2=\mathbb R^2/\mathbb Z^2$ is also periodic with a rotation vector $(1,1)$. 
Let $T_\xi>0$ denote the minimum positive period of the
projected orbit on $\mathbb{T}^2$; equivalently,
\[
\xi(t+T_\xi)=\xi(t)+(1,1)
\qquad\text{for all }t\in\mathbb{R}.
\]

\begin{lemma}
Denote by
\[
    \lambda
    =
    \inf_{\psi\in C^1(\mathbb T^2)}
    \sup_{x\in\mathbb T^2}
    \bigl(W(x)\cdot D\psi(x)-g(x)\bigr).
\]
Then,
    \begin{equation}\label{eq:lambda-orbit-formula}
    \lambda
    =
    -\min_{\xi\in \Gamma}
    \frac1{T_\xi}
    \int_0^{T_\xi}g(\xi(t))\,dt.
\end{equation}
\end{lemma}

\begin{proof}
Write
\[
    P(t)=2+a'(t).
\]
Then, $1\leq P(t)\leq 3$ and $Du(x)=(P(x_1),P(x_2))$.
Thus,
\[
    \sqrt{1+|Du|^2}=\sqrt{1+P(x_1)^2+P(x_2)^2},\qquad
    W(x)
    =
    \frac{(P(x_1),P(x_2))}
    {\sqrt{1+P(x_1)^2+P(x_2)^2}}.
\]
Define
\[
    F(s)=\int_0^s \frac{d\tau}{P(\tau)}.
\]
Then, $F'(s)=1/P(s)$, and
\[
    F(s+1)-F(s)=L=\int_0^1\frac{d\tau}{P(\tau)}.
\]
In addition, the function $K(x_1,x_2)=F(x_1)-F(x_2)$ is constant along each curve \(\xi\in\Gamma\) since $DK\cdot \dot \xi=0$. Therefore, for each \(r\in\mathbb R\), the level set
\[
    C_r:=\{x=(x_1,x_2)\in\mathbb R^2:\ F(x_1)-F(x_2)=r\}
\]
coincides with a curve in \(\Gamma\). Consequently, the family $\{C_r\}_{r\in\mathbb R}$ foliates \(\mathbb R^2\), $C_r+(1,1)=C_r$, and for all $k\in \mathbb{Z}$,
\begin{equation}\label{eq:curve-period}
C_{r+kL}=C_r+k(1,0).
\end{equation}

A direct computation gives the following explicit formula for \(g\):
\begin{align}\label{eq:g-explicit}
    g(x)
    =\frac{\Delta u(1+|Du|^2)-Du\cdot D^2u\cdot Du}{1+|Du|^2}
    =
    \frac{
        P'(x_1)\bigl(1+P(x_2)^2\bigr)
        +
        P'(x_2)\bigl(1+P(x_1)^2\bigr)
    }{
        1+P(x_1)^2+P(x_2)^2
    }.
\end{align}

\noindent {\bf Step 1.}  First, we prove the lower bound. Choose $\xi\in \Gamma$. For any \(\psi\in C^1(\mathbb T^2)\),
\[
    \frac{d}{dt}\psi(\xi(t))
    =
    W(\xi(t))\cdot D\psi(\xi(t)).
\]
Then,
\[
    \int_0^{T_\xi}
    W(\xi(t))\cdot D\psi(\xi(t))\,dt
    =
    0.
\]
Hence,
\[
    \sup_{\mathbb T^2}(W\cdot D\psi-g)
    \geq
    \frac1{T_\xi}
    \int_0^{T_\xi}
    \bigl(W\cdot D\psi-g\bigr)(\xi(t))\,dt 
    =
    \frac1{T_\xi}
    \int_0^{T_\xi}
    -g(\xi(t))\,dt.
\]
Taking the maximum over all periodic orbits gives
\[
    \sup_{\mathbb T^2}(W\cdot D\psi-g)
    \geq
    \max_{\xi\in \Gamma}
    \frac1{T_\xi}
    \int_0^{T_\xi}-g(\xi(t))\,dt.
\]
Now, take the infimum over \(\psi\) to imply
\[
    \lambda
    \geq
    \max_{\xi\in \Gamma}
    \frac1{T_\xi}
    \int_0^{T_\xi}-g(\xi(t))\,dt
    =-\min_{\xi\in \Gamma}
    \frac1{T_\xi}
    \int_0^{T_\xi}g(\xi(t))\,dt.
\]

\noindent {\bf Step 2.} We prove the reverse inequality. Introduce coordinates $(r,\theta)$ adapted to the flow: 
\[
\begin{cases}
    r=F(x_1)-F(x_2),\\
    \theta=F(x_1),
\end{cases}
\qquad \Longleftrightarrow \qquad
\begin{cases}
    x_1=G(\theta),\\
    x_2=G(\theta-r),
\end{cases}
\]
where $G=F^{-1}$.
In these coordinates, the vector field has the form
\[
    W\cdot D=A(r,\theta)\partial_\theta,\qquad
\text{where} \qquad
    A(r,\theta)=\frac{1}{\sqrt{1+|Du(r,\theta)|^2}}.
\]
For  each $r\in \mathbb{R}$, let $\xi_r\in \Gamma$ be the orbit that coincides with $C_r$. Let $T_{\xi_r}$ be its minimum positive period. 
Let
$\Phi(r,\theta)=\bigl(G(\theta),G(\theta-r)\bigr)$
denote the inverse coordinate map, and, for a function
\(h:\mathbb{R}^2\to\mathbb{R}\), write
\[
\widehat h(r,\theta):=(h\circ\Phi)(r,\theta)
=h\bigl(G(\theta),G(\theta-r)\bigr).
\]
Normalize the lift \(\xi_r\) by requiring that \(\theta(0)=0\).
Since the orbit \(\xi_r\) lies in \(C_r\), its \(r\)-coordinate is
constant, and hence
\[
\xi_r(t)=\Phi\bigl(r,\theta(t)\bigr).
\]
Moreover,
\[
\dot{\theta}(t)
=D\theta(\xi_r(t))\cdot\dot{\xi}_r(t)
=A\bigl(r,\theta(t)\bigr)>0.
\]
After one period on \(\mathbb{T}^2\), the lifted orbit advances by
\((1,1)\). Therefore,
\[
\theta(T_{\xi_r})-\theta(0)
=F\bigl(x_1(0)+1\bigr)-F\bigl(x_1(0)\bigr)
=L.
\]
Consequently, using $dt=\frac{d\theta}{A(r,\theta)}$,
we obtain
\[
\int_0^{T_{\xi_r}} h(\xi_r(t))\,dt
=
\int_0^{T_{\xi_r}}
\widehat h\bigl(r,\theta(t)\bigr)\,dt
=
\int_0^L
\frac{\widehat h(r,\theta)}{A(r,\theta)}\,d\theta.
\]
Hereafter, by a harmless abuse of notation, we write \(h(r,\theta)\)
in place of \(\widehat h(r,\theta)\).
Due to \eqref{eq:curve-period}, the above is a smooth $L$-periodic function of $r$ if $h=h(x)\in C^\infty(\T^2)$. 
For each $r\in \mathbb{R}$, define the time average of \(-g\) along $\xi_r$:
\begin{equation}\label{eq:mr}
  m(r)
    =
    \frac1{T_{\xi_r}}
    \int_0^{T_{\xi_r}} -g(\xi_r(t))\,dt.   
\end{equation}
Then, $m(r)$ is smooth, $L$-periodic, and satisfies  
\[
    \int_0^L
    \frac{m(r)+g(r,s)}{A(r,s)}\,ds=0.
\]
Now define
\begin{equation}\label{eq:psi}
     \psi(r,\theta)
    =
    \int_0^\theta
    \frac{m(r)+g(r,s)}{A(r,s)}\,ds.
\end{equation}
Then, \(\psi\) is also smooth and  $L$-periodic in both \(r\) and \(\theta\). Accordingly, \(\psi\) defines a smooth periodic function on \(\mathbb R^2\).
Moreover,
\[
    W\cdot D\psi
    =
    A(r,\theta)\partial_\theta\psi
    =
    m(r)+g(r,\theta).
\]
Therefore,
\[
    W\cdot D\psi-g=m(r) \qquad \text{in } \mathbb{R}^2.
\]
Hence,
\[
    \sup_{\mathbb T^2}(W\cdot D\psi-g)
    =
    \max_r m(r)
    =
    \max_{\xi\in \Gamma}
    \frac1{T_\xi}
    \int_0^{T_\xi}-g(\xi(t))\,dt.
\]
Thus,
\[
    \lambda
    \leq
    \max_{\xi \in \Gamma}
    \frac1{T_\xi}
    \int_0^{T_\xi}-g(\xi(t))\,dt.
\]

Combining the two inequalities gives the conclusion. 
    
\end{proof}

\begin{lemma}\label{lem:limit}
    \[
    \lim_{d\to 0^+}\frac{c(d)}{d}=\lambda. 
    \]
\end{lemma}

\begin{proof}
Define the nonlinear operator
\begin{equation*}
    H_d[\phi](x)=
    \left(
        1-d\diver\left(
        \frac{p+D\phi}{\sqrt{1+|p+D\phi|^2}}
        \right)
    \right)
    \sqrt{1+|p+D\phi|^2}+f(x).    
\end{equation*}
Let $v_d$ be a solution to the corresponding cell problem, that is,
\[
    H_d[v_d](x)=c(d) \qquad\text{in }\mathbb T^2.
\]

\noindent {\bf Step 1.} We first establish the upper bound.  Take \(\psi\in C^\infty(\T^2)\), and consider
\[
    \phi=v+d\psi.
\]
Then,
\[
    p+D\phi=Du+dD\psi,
\]
and we have the uniform expansion
\[
    \sqrt{1+|Du+dD\psi|^2}
    =
    \sqrt{1+|Du|^2}+d\,W\cdot D\psi+O(d^2).
\]
Also,
\[
    \frac{Du+dD\psi}{\sqrt{1+|Du+dD\psi|^2}}
    =
    W+O(d),
\]
and hence
\[
    \diver\left(
    \frac{Du+dD\psi}{\sqrt{1+|Du+dD\psi|^2}}
    \right)
    =
    \diver W+O(d).
\]
Therefore
\begin{equation}\label{eq:linear-expansion}
    H_d[v+d\psi]
    \leq d\max_{\T^2}\bigl(W\cdot D\psi-g\bigr)+O(d^2).
\end{equation}
Choosing a maximum point $x_1$ of  $(v+d\psi)-v_d$, we have that
\begin{align*}
c(d)=H_d[v_d](x_1) \leq H_d[v+d\psi](x_1)\leq d\max_{\T^2}\bigl(W\cdot D\psi-g\bigr)+O(d^2).
\end{align*}
Since $\psi$ is arbitrary, this shows that 
\[
    \limsup_{d\to0^+}\frac{c(d)}d\leq \lambda.
\]

\noindent {\bf Step 2.} Next, we prove the lower bound,  which requires a slightly more careful perturbed test function method. The reason is that the transport equation
\[
    W\cdot D\psi-g=\lambda
\]
may not be solvable globally unless all orbit averages of \(g\) agree.
Let
\[
    r=F(x_1)-F(x_2).
\]
$m(r)$ and $\psi$ be from \eqref{eq:mr} and \eqref{eq:psi}, respectively. 
Then, $\lambda=\max_r m(r)$
and
\begin{equation}\label{eq:orbit-transport}
    W\cdot D\psi-g=m(r) \qquad \text{in } \mathbb{R}^2.
\end{equation}

Let \(\chi=\chi(r)\) be a smooth $L$-periodic function to be chosen. Since \(Du\cdot Dr=0\), we have
\[
   Du\cdot D\chi=W\cdot D\chi=0 \qquad \text{in } \mathbb{R}^2.
\]
Consider the perturbed function
\[
    \phi_d=v+\sqrt d\,\chi(r)+d\psi.
\]
Then,
\[
    p+D\phi_d
    =
    Du+\sqrt d\,D\chi+dD\psi,
\]
\[
    |p+D\phi_d|^2=|Du|^2+d|D\chi|^2+2dD\psi\cdot Du+O(d^{3/2}),
\]
and 
\begin{equation}\label{eq:lower-expansion}
    H_d[\phi_d]
    =
    d\left(
        W\cdot D\psi-g+\frac{|D\chi|^2}{2\sqrt{1+|Du|^2}}
    \right)
    +O(d^{3/2}).
\end{equation}
Using \eqref{eq:orbit-transport}, this becomes
\[
    H_d[\phi_d]
    =
    d\left(
        m(r)+\frac{|D\chi|^2}{2\sqrt{1+|Du|^2}}
    \right)
    +O(d^{3/2}).
\]
Since $|D\chi|\geq \tfrac{\sqrt{2}}{3}|\chi'(r)|$, for any fixed $\varepsilon>0$, we could 
choose \(\chi=\chi(r)\) so that
\begin{equation}\label{eq:approx-ineq}
    m(r)+\frac{|D\chi(r)|^2}{2\sqrt{1+|Du|^2}}
    \geq \lambda-\varepsilon
    \qquad\text{in }\mathbb T^2.
\end{equation}
Indeed, let
\[
I=\left\{r\in [0,L]\,:\, m(r)>\lambda - \frac{\varepsilon}{2} \right\}.
\]
Construct $\chi' \in C^\infty(\R)$ that is $L$-periodic such that, for $K>1$ sufficiently large,
\[
    \chi'(r) = K \qquad \text{ for } r \in [0,L] \setminus I,\qquad
    \int_0^L \chi'(r)\,dr=0.
\]
Then,  $\chi \in C^\infty(\R)$, which is $L$-periodic.
By taking 
\[
K^2>9\sqrt{1+\|Du\|_{L^\infty}^2}\,\left(\max_r m(r) - \min_r m(r)\right),
\]
we see that \eqref{eq:approx-ineq} holds.
Then, \eqref{eq:lower-expansion} gives
\[
    H_d[\phi_d]
    \geq
    d(\lambda-\varepsilon)+O(d^{3/2}) \qquad\text{in }\mathbb T^2.
\]
Choosing a maximum point $x_2$ of  $v_d-\phi_d$, we have that
\[
    c(d)=H_d[v_d](x_2)\geq H_d[\phi_d](x_2)
    \geq
    d(\lambda-\varepsilon)+O(d^{3/2}).
\]
Dividing by \(d\) and letting \(d\to0^+\), we obtain
\[
    \liminf_{d\to0^+}\frac{c(d)}d
    \geq
    \lambda-\varepsilon.
\]
Since \(\varepsilon>0\) is arbitrary,
\[
    \liminf_{d\to0^+}\frac{c(d)}d
    \geq
    \lambda.
\]
Combining this inequality with the upper bound concludes the proof.

\end{proof}

\begin{proof}[Proof of Theorem \ref{theo:main}]
By Lemma \ref{lem:limit},
    \[
    \lim_{d\to 0^+}\frac{c(d)}{d}=\lambda. 
    \]
We now choose \(a=a_\varepsilon\) as in Section \ref{sec:appendix1}, with \(\varepsilon>0\) sufficiently small.
Then, by Lemma \ref{lem:positive}, we see that $\lambda>0$. 
The proof is complete.
\end{proof}

\section{An example with a negative orbit average}\label{sec:appendix1}

We give an explicit example showing that the orbit average of \(g\) can be negative along at least one orbit:
\[
\int_{0}^{T_\xi}g(\xi(t))\,dt<0.
\]

The example is based on a small two-mode perturbation. Two modes are
needed to prevent cancellation. Let
\[
a_\delta(t)
=
-\frac{\delta}{2\pi}\cos(2\pi t)
-\frac{\delta}{4\pi}\cos(4\pi t),
\qquad 0<\delta\ll 1.
\]
Then
\[
a_\delta'(t)
=
\delta\bigl(\sin(2\pi t)+\sin(4\pi t)\bigr).
\]
For \(\delta>0\) sufficiently small,
\[
|a_\delta'(t)|<1
\qquad\text{for all }t\in\mathbb R.
\]
Set
\[
h(t)=\sin(2\pi t)+\sin(4\pi t),
\qquad
P_\delta(t)=2+\delta h(t).
\]
Then
\[
Du_\delta(x)
=
\bigl(P_\delta(x_1),P_\delta(x_2)\bigr),
\]
and
\[
W_\delta(x)
=
\frac{\bigl(P_\delta(x_1),P_\delta(x_2)\bigr)}
{\sqrt{1+P_\delta(x_1)^2+P_\delta(x_2)^2}}.
\]
As before, define
\[
F_\delta(s)
=
\int_0^s\frac{d\tau}{P_\delta(\tau)},
\qquad
L_\delta
=
\int_0^1\frac{d\tau}{P_\delta(\tau)}.
\]
We consider the half-shifted level curve
\[
C_\delta
=
\left\{
F_\delta(x_1)-F_\delta(x_2)
=
\frac{L_\delta}{2}
\right\}.
\]
Let \(\xi_\delta\) be the associated flow.

\begin{lemma}\label{lem:positive}
When \(\delta>0\) is sufficiently small,
\[
\int_0^{T_{\xi_\delta}}g(\xi_\delta(t))\,dt<0.
\]
\end{lemma}

\begin{proof}
For a given \(x_1\), let \(x_2=y(x_1)\) be determined by
\[
F_\delta(x_1)-F_\delta(y(x_1))
=
\frac{L_\delta}{2}.
\]
A direct computation shows that
\[
L_\delta
=
\frac12+O(\delta^2)
\]
and
\begin{equation}\label{eq:y-expansion}
y(x_1)
=
x_1-\frac12
+\delta\eta_1(x_1)
+\delta^2\eta_2(x_1)
+O(\delta^3),
\end{equation}
where
\[
\eta_1(x_1)
=
\frac{\cos(2\pi x_1)}{2\pi},
\]
and
\[
\eta_2(x_1)
=
\frac{
\bigl(8\sin^4(\pi x_1)-2\sin^2(\pi x_1)+3\bigr)
\sin(\pi x_1)\cos(\pi x_1)
}{6\pi}.
\]
Recall that
\begin{equation}\label{eq:g-explicit-two-mode}
g(x_1,x_2)
=
\frac{
P_\delta'(x_1)\bigl(1+P_\delta(x_2)^2\bigr)
+
P_\delta'(x_2)\bigl(1+P_\delta(x_1)^2\bigr)
}{
1+P_\delta(x_1)^2+P_\delta(x_2)^2
}.
\end{equation}
Along the orbit,
\[
\frac{dx_1}{dt}
=
\frac{P_\delta(x_1)}
{\sqrt{1+P_\delta(x_1)^2+P_\delta(y(x_1))^2}}
>0.
\]
Therefore,
\begin{equation}\label{eq:gdt}
\begin{aligned}
g(\xi(t))\,dt
&=
E(\delta,x_1)\,dx_1\\
&=
\frac{
P_\delta'(x_1)\bigl(1+P_\delta(y(x_1))^2\bigr)
+
P_\delta'(y(x_1))\bigl(1+P_\delta(x_1)^2\bigr)
}{
P_\delta(x_1)
\sqrt{1+P_\delta(x_1)^2+P_\delta(y(x_1))^2}
}\,dx_1.
\end{aligned}
\end{equation}
Thus, by \eqref{eq:y-expansion},
\[
g(\xi(t))\,dt
=
\left(
\delta A_1(x_1)
+\delta^2A_2(x_1)
+\delta^3A_3(x_1)
+O(\delta^4)
\right)\,dx_1,
\]
where, for \(j=1,2,3\),
\[
A_j(x_1)
=
\frac1{j!}\partial_\delta^jE(0,x_1).
\]
More explicitly,
\begin{align*}
A_1&=\frac56(H_1+K_1),\\
A_2&=
\frac{
90\eta_1K_2
-65HH_1
+7HK_1
+52H_1K
-20KK_1
}{108},\\
A_3&=
\frac1{648}\Bigl(
270\eta_1^2K_3
+42\eta_1HK_2
+312\eta_1H_1K_1
-120\eta_1KK_2
-120\eta_1K_1^2\\
&\qquad\qquad
+540\eta_2K_2
+205H^2H_1
+H^2K_1
-172HH_1K
+44HKK_1\\
&\qquad\qquad
+22H_1K^2
+10K^2K_1
\Bigr),
\end{align*}
where
\begin{align*}
&H=h(x_1),
\qquad
H_1=h'(x_1),\\
&K=h(x_1-0.5),
\qquad
K_1=h'(x_1-0.5),\\
&K_2=h''(x_1-0.5),
\qquad
K_3=h'''(x_1-0.5).
\end{align*}
The first two coefficients have zero average:
\[
\int_0^1A_1(x_1)\,dx_1=0,
\qquad
\int_0^1A_2(x_1)\,dx_1=0.
\]
The cubic coefficient satisfies
\[
\int_0^1A_3(x_1)\,dx_1
=
-\frac{2\pi}{3}.
\]
Consequently,
\[
\begin{aligned}
\int_0^{T_{\xi_\delta}}g(\xi_\delta(t))\,dt
&=
\int_0^1
\left(
\delta A_1(x_1)
+\delta^2A_2(x_1)
+\delta^3A_3(x_1)
\right)\,dx_1
+O(\delta^4)\\
&=
-\frac{2\pi}{3}\delta^3
+O(\delta^4).
\end{aligned}
\]
Therefore, for all sufficiently small \(\delta>0\),
\[
\int_0^{T_{\xi_\delta}}
g(\xi_\delta(t))\,dt<0.
\]
\end{proof}

\appendix
\section{A supplementary uncutoff cell-problem mechanism}\label{sec:appendix2}

In this section, we record another construction showing that $c(d)$ need not be monotonically decreasing in $d$. This example is based on a different mechanism from the one used in the main text. It was also developed with substantial assistance from ChatGPT 5.5 Plus, which helped suggest the main perturbative steps; the computations and verification were then checked and completed by the authors.

We work on $\mathbb{T}^2=\mathbb{R}^2/(2\pi\mathbb{Z})^2$ and use the normalized average
\[
    \langle h\rangle=\frac{1}{(2\pi)^2}\int_{\mathbb{T}^2}h(y)\,dy.
\]
Consider the forced graph mean curvature cell problem
\begin{equation}\label{eq:appendix-cell}
    \left(1-d\,\operatorname{div}
    \frac{Du}{\sqrt{1+|Du|^2}}\right)
    \sqrt{1+|Du|^2}+f_N(x)=c_N(d)
    \qquad \text{in }\mathbb{T}^2 .
\end{equation}
Let
\[
    \phi(y_1,y_2)=\cos y_1+\cos y_2-\cos(y_1+y_2),
\]
and fix $L>2$. Define
\[
    F(y)=L\Delta \phi(y)
    =
    L\bigl(-\cos y_1-\cos y_2+2\cos(y_1+y_2)\bigr),
\]
and, for a large integer $N$,
\[
    f_N(x)=F(Nx).
\]
We claim that, for all sufficiently large $N$, the corresponding ergodic constant satisfies
\[
    c_N'(1)>0.
\]
Thus, $d\mapsto c_N(d)$ is locally increasing near $d=1$.

The idea is to introduce the high-frequency scaling
\[
    y=Nx,\qquad \mu=N^{-2},\qquad u(x)=\mu U(Nx).
\]
Then, $D_xu=N^{-1}D_yU$ and $D_x^2u=D_y^2U$. Writing derivatives with respect to $y$, \eqref{eq:appendix-cell} becomes
\begin{equation}\label{eq:appendix-scaled}
    -d\left(I-\frac{\mu DU\otimes DU}{1+\mu|DU|^2}\right):D^2U
    +\sqrt{1+\mu|DU|^2}+F(y)=c .
\end{equation}
At $\mu=0$, this reduces to
\[
    -d\Delta U+1+F(y)=c.
\]
Since $F=L\Delta\phi$, the zero-average solution is
\[
    U_0(y;d)=\frac{L}{d}\phi(y),
    \qquad c(0,d)=1.
\]
By the implicit function theorem, for $\mu$ sufficiently small and $d$ near $1$, there is a smooth branch
\[
    U_{\mu,d}=U_0+\mu U_1+O(\mu^2),
    \qquad
    c(\mu,d)=1+\mu c_1(d)+O(\mu^2),
\]
and, by uniqueness of the ergodic constant,
\[
    c_N(d)=c(N^{-2},d).
\]

Expanding \eqref{eq:appendix-scaled} to first order in $\mu$ gives
\[
    -d\Delta U_1
    +d(DU_0\otimes DU_0):D^2U_0
    +\frac12 |DU_0|^2
    =
    c_1(d).
\]
Averaging over $\mathbb{T}^2$ eliminates the Laplacian term, hence
\[
    c_1(d)
    =
    \left\langle
    d(DU_0\otimes DU_0):D^2U_0
    +\frac12 |DU_0|^2
    \right\rangle .
\]
Using $U_0=(L/d)\phi$, we obtain
\[
    c_1(d)
    =
    \frac{L^3}{d^2}
    \left\langle
    (D\phi\otimes D\phi):D^2\phi
    \right\rangle
    +
    \frac{L^2}{2d^2}
    \left\langle |D\phi|^2\right\rangle .
\]
A direct Fourier computation gives
\[
    \left\langle |D\phi|^2\right\rangle=2,
    \qquad
    \left\langle
    (D\phi\otimes D\phi):D^2\phi
    \right\rangle
    =-\frac12 .
\]
Therefore
\[
    c_1(d)
    =
    \frac{L^2}{d^2}
    \left(1-\frac{L}{2}\right).
\]
Since $L>2$, this coefficient is negative. Consequently,
\[
    c_N(d)
    =
    1+\frac{1}{N^2}\frac{L^2}{d^2}
    \left(1-\frac{L}{2}\right)
    +O(N^{-4}),
\]
where the remainder is $C^1$ in $d$ for $d$ near $1$. Differentiating at $d=1$ yields
\[
    c_N'(1)
    =
    -\frac{2}{N^2}L^2\left(1-\frac{L}{2}\right)
    +O(N^{-4})
    =
    \frac{L^2(L-2)}{N^2}
    +O(N^{-4})>0
\]
for all sufficiently large $N$. Hence $c_N(d)$ is locally increasing near $d=1$, and the monotone-decreasing behavior in $d$ also fails away from $d=0$.

\bibliographystyle{alpha}
\bibliography{pnas-TXY}

\end{document}